\documentclass{article}

\usepackage[utf8]{inputenc}
\usepackage[english]{babel}
\usepackage{fullpage}

\def\lc{\left\lceil}   
\def\rc{\right\rceil}
\def\lf{\left\lfloor}   
\def\rf{\right\rfloor}

\usepackage{amsmath,amsthm,amssymb}
\usepackage{hyperref}
\usepackage{enumerate}

\newtheorem{theorem}{\noindent Theorem}
\newtheorem{lemma}{\noindent Lemma}

\newtheorem{example}{\noindent Example}

\title{On the Erd{\H o}s--Hajnal problem in the case of 3-graphs}

\author{Danila Cherkashin\footnote{National Research University Higher School of Economics, Soyuza Pechatnikov str., 16, St. Petersburg, Russian Federation}}

\begin{document}

\maketitle

\begin{abstract}
Let $m(n,r)$ denote the minimal number of edges in an $n$-uniform hypergraph which is not $r$-colorable. 
For the broad history of the problem see~\cite{RaiSh}. 
It is known~\cite{cherkashin2018regular} that for a fixed $n$ the sequence
\[
\frac{m(n,r)}{r^n} 
\]
has a limit. 

The only trivial case is $n=2$ in which $m(2,r) = \binom{r+1}{2}$. In this note we focus on the case $n=3$.
First, we compare the existing methods in this case and then improve the lower bound.
\end{abstract}

\section{Introduction}

A hypergraph $H=(V,E)$ consists of a finite set of \textit{vertices}
$V$ and a family $E$ of the subsets of $V$,
which are called \textit{edges}. A hypergraph is called
\textit{$n$-uniform} if every edge has size $n$.
A \textit{vertex $r$-coloring} of a hypergraph $H = (V,E)$ is a map from $V$ to $\{1,\ldots,r\}$.
A coloring is \textit{proper} if there is no monochromatic edges, i.e., any
edge $e\in E$ contains two vertices
of different color.
The \textit{chromatic number} of a hypergraph $H$ is the smallest number $\chi(H)$ such that there exists a proper $\chi(H)$-coloring of $H$.
Let $m(n,r)$ be the minimal number of edges in an $n$-uniform hypergraph with chromatic number more than $r$.

We are interested in the case when $n$ is much smaller than $r$ (see~\cite{RaiSh} for general case and related problems).
Erd{\H o}s and Hajnal~\cite{erdHos1961property} introduced problems on determining $m(n,r)$ and related quantitites.

\subsection{Upper bounds}

%and Hajnal 

Erd{\H o}s conjectured~\cite{erdos1979some} that 
\[
m(n,r) = \binom{(n-1)r+1}{n},
\]
for $r > r_0(n)$, that
is achieved on the complete hypergraph.

However Alon~\cite{alon1985hypergraphs} disproved the conjecture by using the
estimate
\[
m(n,r) < \min_{a \geq 0} T(r(n + a - 1) + 1, n + a, n),
\]
where the Tur{\'a}n number $T(v,k,n)$ is the smallest number of edges in an $n$-uniform hypergraph on $v$ vertices
such that every induced subgraph on $k$ vertices contains an edge. 
Different bound on Tur{\'a}n numbers give better refine the complete hypergraph construction when $n > 3$ (see~\cite{sidor} for a survey).
So the case $ n =3$ is in some sense the most interest.

Also note, that using the same inequality with better bounds on Tur{\'a}n numbers Akolzin and Shabanov~\cite{akolzin2016colorings} showed that
\[
m(n,r) < Cn^3\ln n \cdot r^n.
\]

Alon conjectured that the sequence $m(n,r)/r^n$ has a limit which was proved by Cherkashin and Petrov~\cite{cherkashin2018regular}.
Denote the corresponding limit be $L_n$. In this paper we are interested in estimates on $L_3$. The best known upper bound follows from the complete hypergraph:
\[
L_3 \leq \frac{4}{3}.
\]

\subsection{Lower bounds}

There are several ways to show an inequality of type $m(n,r) \geq c(n) r^n$ (i.e. $L_n \geq c(n)$). 
Note that Erd{\H o}s--Hajnal conjecture implies in particular that 
\[
L_n = \frac{(n-1)^n}{n!}.
\]

Alon~\cite{alon1985hypergraphs} suggested to color vertices of an $n$-uniform hypergraph in $a < r$ colors uniformly and independently, 
and then recolor a vertex in every monochromatic edge in unused color.
The expected number of monochromatic edges is 
\[
|E| \cdot a^{1-n}.
\]
Note that we have $r-a$ remaining colors, and we can color $n-1$ vertices in each unused color such that no new monochromatic edge appears.
Summing up, if
\[
|E| < a^{n-1} (r-a)(n-1)
\]
then a hypergraph $H = (V,E)$ has a proper $r$-coloring.
Substituting $a = \lf\frac{n-1}{n}r\rf$, we get
\[
m(n,r) \geq (n-1) \lc \frac{r}{n} \rc \lf \frac{n-1}{n}r \rf ^{n-1}.
\]
This method gives $L_3 \geq 8/27 > 0.296$.

Another way is due to Pluh{\'a}r~\cite{Pl}. He introduced the following useful notion.
A sequence of edges $a_1,\ldots, a_r$ is an \textit{$r$-chain} if $|a_i\cap a_j| = 1$ if $|i - j| = 1$ and   $a_i \cap a_j = \emptyset$ otherwise;
it is an \textit{ordered $r$-chain} if $i < j$ implies that every vertex of $a_i$ is not bigger than any vertex of $a_j$
(with respect to a certain fixed linear ordering on $V$).

Pluh{\'a}r's theorem states that existence of an order on $V$ without ordered $r$-chains is equivalent to $r$-colorability of $H = (V,E)$.
Let us prove a lower bound on $m(n,r)$ via this theorem. Consider a random order on the vertex set. 
Note that the probability of an $r$-chain to be ordered is 
\[
\frac{[(n-1)!]^2[(n-2)!]^{r-2}}{((n-1)r+1)!}.
\]
From the other hand, the number of $r$-chains is at most $2|E|^r/r!$ since every set of $r$ edges generates at most 2 chains.
So if 
\[
2\frac{|E|^r}{r!} \frac{[(n-1)!]^2[(n-2)!]^{r-2}}{((n-1)r+1)!} < 1,
\]
then we have a proper $r$-coloring of $H$.
After taking $r$-root and some calculations we have
\[
m(n,r) > c\sqrt{n} r^n,
\]
and in particular $L_3 \geq 4/e^3 > 0.199$.

Combining two previous arguments with Cherkashin--Kozik approach~\cite{cherkashin2015note} Akolzin and Shabanov~\cite{akolzin2016colorings} proved that
\[
m(n,r) \geq c\frac{n}{\ln n}r^n,
\]
without explicit bounds on $c$. We show that this method gives the bound $L_3 > 0.205$ in Section 3.

Cherkashin and Petrov~\cite{cherkashin2018regular} suggested an approach, based on the evaluation of the inverse function, to show that the sequence $m(n,r)/r^n$ has a limit. Denote by $f(N)$ the maximal possible chromatic number of an $n$-uniform hypergraph with $N$ edges. 
Also $f(0)=1$ by agreement. The function $f$ non-strictly increases and 
satisfies
\[
m(n,r)=\min\{N: f(N) > r\}.
\]
Therefore $m(n,r) \sim C r^n$ if and only if $f(N) \sim (N/C)^{1/n}$. 
The following lemmas were proved in~\cite{cherkashin2018regular}.

\begin{lemma}\label{cru}
For any $N > 0$ and any positive integer $p$ we have 
\[
f(N)\leqslant \max_{a_1+a_2+\dots +a_p\leqslant N/p^{n-1}} f(a_1)+f(a_2)+\dots+f(a_p).
\]
\end{lemma}

\begin{lemma}\label{segment}
Denote $c_n=(1-2^{1-n})^{-n}$.
For any
$M>0$ the inequality
\[
f(N)\leqslant 
N^{1/n}\cdot \max_{M\leqslant a< c_n M} f(a)\cdot a^{-1/n}
\]
holds for all $N\geqslant M$.
\end{lemma}

It is known that $f(0) = 1$, $f(1) = \ldots = f(6) = 2$, $f(7) = \ldots = f(26) = 3$ (see~\cite{akolzin20183}).
Lemmas~\ref{cru},~\ref{segment} and computer calculations was used to get 
\[
L_3 > 0.324.
\]

One more way to get a (very weak) bound $L_n \geq n^{-n}$ appeared in the same paper~\cite{cherkashin2018regular}. 
It is based on the straightforward induction on $r$: it was shown that there is a large independent set, so we can color it in color $r$ and apply the inductive assumption.

The contribution of the paper is the following theorem, which is proved by refining Pluhar approach via inducibility arguments.
\begin{theorem}
\[
L_3 \geq \frac{4}{e^2} > 0.54.
\]
\label{hahaha}
\end{theorem}

\paragraph{Structure of the paper.} In Section~\ref{ind} we show how to apply inducibility to the chain argument and proof Theorem~\ref{hahaha}.
In Section 3 we find the constant in Akolzin--Shabanov theorem for $n = 3$ and show that even if we apply Theorem~\ref{main} to the corresponding part of the proof, the constant will be still worse than in Theorem~\ref{hahaha}.

\section{Inducibility tool}
\label{ind}

\begin{theorem}
\label{main}
Suppose $H = (V,E)$ is a hypergraph. Then it has at most
\[
\frac{|E|}{2}\left ( \frac{|E|}{r-1} \right )^{r-1}
\]
$r$-chains.
\end{theorem}

We need a notion of inducibility. Denote by $I(G,H)$ the number of induced subgraphs in $G$, isomorphic to $H$.
The following simple bound was proved by Pippenger and Golumbic.

\begin{lemma}[Pippenger--Golumbic~\cite{pippenger1975inducibility}]
\label{PG}
Let $G$ be a graph on $n$ vertices. Then
\[
I(G, P_r) \leq \frac{n}{2} \left ( \frac{n}{r-1} \right )^{r-1}.
\]
\end{lemma}

%From the other hand there is an example on
%\[
%\frac{n^r}{(r+1)^{r-1} - (r+1)}
%\]
%vertices. This is the blow-up of $C_{r+1}$ where the same construction is placed inside the blow-up of each vertex.

It turns out that the bound is close to optimal.
\begin{example}
Let $n$ be the $k$-th power of $r+1$, $r > 2$.
Consider the blow-up of $C_{r+1}$ where the same construction is placed inside the blow-up of each vertex. Every $r$-chain 
has vertices only in the different copies. Hence the number of chains is
\[
\sum_{i=1}^{k} \frac{n^r}{(r+1)^{i(r-1)}}.
\]
\end{example}

\begin{proof}[Proof of Lemma~\ref{PG}]
Let $X(q, l)$ denote the largest possible number of ways of sequentially
choosing $q$ objects $w_0, w_1 , \dots, w_{q-1}$ from among $l$ objects, subject to rules
whereby the set of objects that are eligible to be chosen as $w_i$ depends only
on the previous choices $w_0, w_1, \dots, w_{i-1}$, and whereby no object that is
eligible to be chosen as $w_i$ will be eligible to be chosen as $w_j$ for any
$i + 1 \leq j \leq q - 1$. Clearly, $X(0, l) = 1$. If $q > 0$, let $m$ denote the
number of objects eligible to be chosen as $w_0$. For any choice of $w_0$,
the remaining $q - 1$ objects can be chosen in at most $X(q - 1, l - m)$ ways. Thus
\[
X (q, l) \leq \max_{1\leq m \leq l} m X(q-1, l-m-1) \leq \max_{1\leq m \leq l} m X(q-1, l-m).
\]
From these relations, we obtain
\[
X(q,l) \leq \left ( \frac{l}{q} \right )^q.
\]
by induction on $q$: the base $q = 1$ is obvious. To prove the step it is enough to maximize the right-hand side of
\[
X (q, l) \leq  \max_{1\leq m \leq l} m \left (\frac{l-m}{q-1} \right )^{q-1}.
\]
Taking derivative, we get the maximum at $m = l/q$, and we are done.

Also there are $n$ ways to choose the first vertex.
Obviously, it is an upper bound on the number of induced $r$-paths, multiplied by 2, because every path is counted twice.
\end{proof}

\begin{proof}[Proof of Theorem~\ref{main}]
Consider an auxiliary graph $G = (E, F)$ with vertex set is the edge set of $H$ and edges connect pairs of graph vertices, which intersect (as hyperedges) on exactly one vertex.
The number of $r$-chains is at most the number of induced graphs $P_r$ in $G$, because every $r$-chain forms induced $P_r$ (note that the reverse consequence is wrong,
because non-edge in $G$ can mean that the corresponding hyperedges have large intersection, which is impossible in $r$-chain).
Hence, Lemma~\ref{PG} finishes the proof.
\end{proof}

\begin{proof}[Proof of Theorem~\ref{hahaha}]
Let us try to color by the Pluhar greedy algorithm.
Recall that the probability of an $r$-chain to be ordered is 
\[
\frac{[(n-1)!]^2[(n-2)!]^{r-2}}{((n-1)r+1)!} = \frac{4}{(2r+1)!}.
\]
Using Theorem~\ref{main} we get that if
\[
\frac{|E|^r}{2(r-1)^{r-1}}\frac{4}{(2r+1)!} < 1,
\]
than hypergraph is $r$-colorable.
Summing up,
\[
L_3 \geq \lim\limits_{r\to \infty} \sqrt[r]{\frac{(2r+1)!(r-1)^{r-1}}{2}} \frac{1}{r^3} = \frac{4}{e^2}.
\]
\end{proof}

\section{Analysis of the Akolzin--Shabanov proof}

We rewrite the proof from~\cite{akolzin2016colorings} to get optimal constant in the case $n=3$.

First, for every vertex $v$ introduce the weight $w(v)$ as randomly (accordingly to the uniform distribution and independently) chosen number from $[0,1]$.
Fix parameters $p \in [0,1]$, $a < r$. An edge $e$ is called \textit{bad} if
\[
\max_{v \in e} w(v) - \min_{v \in e} w(v) \leq \frac{1 - p}{a}; 
\]
otherwise it is called \textit{good}.

The coloring algorithm is the following. First we color a (random) subhypergraph, consisting of all good edges, in $a$ colors via Pluhar approach; then 
we color (or recolor) some vertices from bad edges in unused $r-a$ colors.
If Pluhar approach succeed (i.e. there is no ordered $a$-chains) and we have at most $(n-1)(r-a)$ bad edges, then the algorithm return a proper $r$-coloring.
Let us evaluate the probability of success.

\begin{lemma}[Akolzin--Shabanov~\cite{akolzin2016colorings}]
\[
P \left [ e \mbox{ is bad} \right] = \left ( \frac{1-p}{a}\right )^{n-1} \left (\frac{1-p}{a} + n\left ( 1 - \frac{1-p}{a}  \right ) \right ) \leq
n \left ( \frac{1-p}{a}\right )^{n-1} = 3\left ( \frac{1-p}{a}\right )^{2}.
\]
\end{lemma}

Let $C(A_1, \dots , A_a)$ denote the event that all the edges $A_j$ are good and $(A_1, \dots , A_a)$ is an ordered $a$-chain.
\begin{lemma}[Akolzin--Shabanov~\cite{akolzin2016colorings}]
\[
P \left [ C(A_1, \dots , A_a) \right] \leq   a^{-a(n-2)} \frac{p^{a-1}}{(a-1)!} = a^{-a} \frac{p^{a-1}}{(a-1)!}.
\]
\end{lemma}
By Theorem~\ref{ind} we have at most $(|E|/a)^a$ $a$-chains. Define $c = |E|/r^3$; we need
\[
\left ( \frac{|E|}{a} \right )^a a^{-a} \frac{p^{a-1}}{(a-1)!} = \left ( (1+o(1)) \frac{|E|pe}{a^{3}} \right )^a
= \left ( (c+o(1)) \frac{r^3pe}{a^{3}} \right )^a < 1.
\]
Also we need at most $(n - 1)(r - a) = 2(r - a)$ bad edges:
\[
P [X > 2(r - a)] \leq \frac{1}{2(r-a)}\frac{3(1-p)^{2}|E|}{a^2} < 1.
\]
Define $x = r/a$. Then we need $cx^3pe < 1$ and 
\[
\frac{3c(1-p)x^3}{2(x-1)} < 1.
\]
Computer simulations gives that for $p = 0.741$ and $x = 1.05$ the algorithm with $c = 0.42$ returns a proper coloring with positive probability, which implies
$L_3 > 0.42$.

If we simply follow the initial proof, the required inequalities are
\[
cx^3pe^2 < 1 \quad \mbox{and} \quad \frac{3c(1-p)x^3}{2(x-1)} < 1.
\]
So pure Akolzin--Shabanov approach gives $L_3 > 0.205$. Both constants are worse than in Theorem~\ref{hahaha}.

\section{Open problems}

\begin{itemize}
    \item First, recall that the Erd{\H o}s conjecture is still open in the case $n = 3$.

    \item Also it is natural to ask if $m(n,r)$ is regular on the first variable, i.e.  
\[
\lim\limits_{n\to \infty} \frac{m(n+1,r)}{m(n,r)} = r?
\]

    \item In the proof of Theorem~\ref{main} we consider an auxiliary graph $G$. 
    The problem is to describe the set of graphs, which may be achieved from an $r$-chromatic $n$-uniform hypergraph.
    
\end{itemize}

\paragraph{Acknowledgements.} I am grateful to Alexander Sidorenko for an introduction in inducibility theory and to 
Fedor Petrov, who noted some deficiencies in the paper. 
Also Alexander Sidorenko pointed to an inattention, relating to the use of Tur{\'a}n numbers.

\bibliographystyle{plain}
\bibliography{main}

%\section*{Appendix. The proof of Akolzin--Shabanov theorem}

\end{document}